\documentclass[reqno, 10pt]{amsart}
\usepackage{amssymb}
\usepackage{amsmath}
\usepackage[mathscr]{euscript}
\usepackage{hyperref}
\usepackage{graphicx}
\usepackage[normalem]{ulem}

\usepackage{amssymb}
\usepackage{hyperref}
\RequirePackage[dvipsnames]{xcolor} 
\definecolor{halfgray}{gray}{0.55} 
\definecolor{webgreen}{rgb}{0,0.5,0}
\definecolor{webbrown}{rgb}{.6,0,0} \hypersetup{%
  colorlinks=true, linktocpage=true, pdfstartpage=3,
  pdfstartview=FitV,%
  breaklinks=true, pdfpagemode=UseNone, pageanchor=true,
  pdfpagemode=UseOutlines,%
  plainpages=false, bookmarksnumbered, bookmarksopen=true,
  bookmarksopenlevel=1,%
  hypertexnames=true,
  pdfhighlight=/O,
  urlcolor=webbrown, linkcolor=RoyalBlue,
  citecolor=webgreen, 
  pdftitle={},%
  pdfauthor={},%
  pdfsubject={2000 MAthematical Subject Classification: Primary:},%
  pdfkeywords={},%
  pdfcreator={pdfLaTeX},%
  pdfproducer={LaTeX with hyperref}%
}

\usepackage{enumitem}

\newtheorem{theorem}{Theorem}[section]
\newtheorem{lemma}[theorem]{Lemma}

\def\P{{\mathbb{P}}}
\def\d{{\mathsf{d}}}

\begin{document}

\title[Continuity of Oseledets subspaces for fiber-bunched cocycles]{A note on the continuity of Oseledets subspaces for fiber-bunched cocycles}

\author{Lucas Backes }

\address{\noindent IME - Universidade do Estado do Rio de Janeiro, Rua S\~ao Francisco Xavier 524, CEP 20550-900, Rio de
  Janeiro, RJ, Brazil . 
\newline e-mail: \rm
  \texttt{lhbackes@impa.br} }

\maketitle

\begin{abstract}
We prove that restricted to the subset of fiber-bunched elements of the space of $GL(2,\mathbb{R})$-valued cocycles Oseledets subspaces vary continuously, in measure, with respect to the cocycle. 
\end{abstract}

\section{Introduction}

In its simple form, a linear cocycle is just an invertible dynamical system $f:M \rightarrow M$ and a matrix-valued map $A:M\rightarrow GL(d, \mathbb{R})$. Sometimes one calls linear cocycle (over $f$ generated by $A$), instead, the sequence $\lbrace A^n\rbrace _{n\in \mathbb{Z}}$ defined by
\begin{equation*}\label{def:cocycles}
A^n(x)=
\left\{
	\begin{array}{ll}
		A(f^{n-1}(x))\ldots A(f(x))A(x)  & \mbox{if } n>0 \\
		Id & \mbox{if } n=0 \\
		A(f^{n}(x))^{-1}\ldots A(f^{-1}(x))^{-1}& \mbox{if } n<0 \\
	\end{array}
\right.
\end{equation*}
for all $x\in M$.

A special class of cocycles is given when the base dynamics $f$ is hyperbolic and the dynamics induced by $A$ on the projective space is dominated by the dynamics of $f$. That is, the rates of contraction and expansion of the cocycle $A$ along an orbit are smaller than the rates of contraction and expansion of $f$. 
Such a cocycle is called \textit{fiber-bunched} (see Section \ref{sec: definitions and statements} for the precise definitions).

Many aspects of fiber-bunched cocycles are rather well understood. For instance, it is known that their cohomology classes are completely characterized by the information on periodic points \cite{Ba, Sa}, generically they have simple Lyapunov spectrum \cite{BV, VianaAlmostAllCocyc} and in the case when $d=2$, Lyapunov exponents are continuous as functions of the coycle \cite{BBB}. In this short note, still in the context of fiber-bunched cocycles, we address the problem of continuity of the Oseledets subspaces. More precisely, we prove that restricted to the subset of fiber-bunched elements of the space of $GL(2,\mathbb{R})$-valued cocycles Oseledets subspaces vary continuously, in measure, with respect to the cocycle. The proof of this result relies on ideas from \cite{BBB} and \cite{BocV}. In a different context a similar statement was recently gotten by \cite{DK}.

\section{Definitions and statements}\label{sec: definitions and statements}

Let $(M,\d)$ be a compact metric space and $f: M \to M $ be a
homeomorphism. Given any $x\in M$ and $\varepsilon >0$, we define the
\emph{local stable} and \emph{unstable sets} of $x$ with respect to $f$ by
\begin{align*}
  W^s_{\epsilon} (x) &:= \left\{y\in M : \d(f^n(x),f^n(y))\leq\epsilon ,\ \forall
    n \geq 0\right\}, \\
  W^u_{\epsilon } (x) &:= \left\{y\in M : \d(f^n(x),f^n(y))\leq\epsilon ,\ \forall
    n \leq 0\right\},
\end{align*}
respectively.

Following \cite{AvilaVianaExtLyapInvPrin}, we say that a homeomorphism $f:M\to M$ is \emph{hyperbolic with
    local product structure} (or just \emph{hyperbolic} for short)
  whenever there exist constants $C_1,\epsilon ,\tau>0$ and $\lambda\in
  (0,1)$ such that the following conditions are satisfied:
  
  \begin{itemize}
  \item $\; \d(f^n(y_1),f^n(y_2)) \leq C_1\lambda^n \d(y_1,y_2)$,
    $\forall x\in M$, $\forall y_1,y_2 \in W^s_{\epsilon } (x)$, $\forall
    n\geq 0$;
    
  \item $\; \d(f^{-n}(y_1), f^{-n}(y_2)) \leq C_1\lambda^n
    \d(y_1,y_2)$, $\forall x\in M$, $\forall y_1,y_2 \in W^u_{\epsilon } (x)$,
    $\forall n\geq 0$;
    
  \item If $\d(x,y)\leq\tau$, then $W^s_{\epsilon }(x)$ and
    $W^u_{\epsilon }(y)$ intersect in a unique point which is denoted by
    $[x,y]$ and depends continuously on $x$ and $y$. This property is called \textit{local product structure}.
  \end{itemize}
  
Fix such an hyperbolic homeomorphism and let $A:M \rightarrow GL(d,\mathbb{R})$ be a $r$-H\"older continuous map. This means that there exists $C_2>0$ such that 
\begin{displaymath}
\Vert A(x)-A(y)\Vert \leq C_2 \d(x,y)^r \; \textrm{for any} \; x, y\in M. 
\end{displaymath}
Let us denote by $H^r(M)$ the space of such $r$-H\"older continuous maps. We endow this space with the $r$-H\"older topology which is generated by norm
\begin{displaymath}
\parallel A \parallel _{r}:= \sup _{x\in M} \parallel A(x)\parallel + \sup _{x\neq y} \dfrac{\parallel A(x)-A(y)\parallel}{d(x,y)^r}.
\end{displaymath}

We say that the cocycle generated by $A$ satisfies the \textit{fiber bunching condition} or that the cocycle is \textit{fiber-bunched} if there exists $C_3>0$ and $\theta <1$ such that 
\begin{displaymath}
\Vert A^n(x)\Vert \Vert A^n(x)^{-1}\Vert \lambda ^{nr}\leq C_3 \theta ^n 
\end{displaymath}
for every $x\in M$ and $n\geq 0$ where $\lambda$ is the constant given in the definition of hyperbolic homeomorphism.

Let $\mu$ be an ergodic $f$-invariant probability measure on $M$ with local product structure. Roughly speaking, the last property means that $\mu$ is locally equivalent to the product measure $\mu ^s \times \mu ^u$ where $\mu ^s $ and $\mu ^u$ are measures on the local stable and unstable manifolds respectively induced by $\mu$ via the local product structure of $f$. Since we are not going to use explicitly this property we just refer to \cite{BBB} for the precise definition. 

It follows from a famous theorem due to Oseledets (see \cite{V2}) that for $\mu$-almost every point $x\in M$ there exist numbers $\lambda _1(x)>\ldots > \lambda _{k}(x)$, and a direct sum decomposition $\mathbb{R}^d=E^{1,A}_{x}\oplus \ldots \oplus E^{k,A}_{x}$ into vector subspaces such that
\begin{displaymath}
A(x)E^{i,A}_{x}=E^{i,A}_{f(x)} \; \textrm{and} \; \lambda _i(x) =\lim _{n\rightarrow \infty} \dfrac{1}{n}\log \parallel A^n(x)v\parallel 
\end{displaymath}
for every non-zero $v\in E^{i,A}_{x}$ and $1\leq i \leq k$. Moreover, since our measure $\mu$ is assumed to be ergodic the \textit{Lyapunov exponents} $\lambda _i(x)$ are constant on a full $\mu$-measure subset of $M$ as well as the dimensions of the \textit{Oseledets subspaces} $E^{i,A}_{x}$. Thus, we will denote by $\lambda ^-(A,\mu)=\lambda _k(x)$ and $\lambda ^+(A,\mu)=\lambda _1(x)$ the \textit{extremal Lyapunov exponents} and by $E^{s,A}_{x}=E^{k,A}_{x}$ and $E^{u,A}_{x}=E^{1,A}_{x}$ the \textit{stable and unstable spaces} respectively. It follows by the Sub-Additive Ergodic Theorem of Kingman (see \cite{K} or \cite{V2}) that the extremal Lyapunov exponents are also given by 
\begin{equation*} 
\lambda ^+(A,\mu)= \lim _{n\rightarrow \infty}\dfrac{1}{n}\log \|A^n(x)\|
\end{equation*}
\vspace{-0.2in}
\begin{flalign}\label{eq: extremal Lyapunov}
& \textrm{and}& 
\end{flalign}
\vspace{-0.2in}
\begin{equation*}
\lambda ^-(A,\mu)= \lim _{n\rightarrow \infty}\dfrac{1}{n}\log \| (A^n(x))^{-1}\|^{-1}
\end{equation*}
for $\mu$ almost every point $x \in M$. The objective of this note is to understand, for a fixed base dynamics $f$, how does the map $A\to E^{i,A}_{x}$ vary in the case when $d=2$, that is, in the case when the cocycle $A$ takes values in $GL(2,\mathbb{R})$.

Let $d$ be the distance on the projective space $\mathbb{P}(\mathbb{R}^2)$ defined by the angle between two directions. We say that an element $A$ of $H^r(M)$ with $\lambda ^+(A,\mu) >\lambda ^-(A,\mu)$ is a \emph{continuity point for the Oseledets decomposition with respect to the measure $\mu$} if the Oseledets subspaces are continuous, in measure, as functions of the cocycle. More precisely, for any sequence $\lbrace (A_k)_{k\in \mathbb{N}}\rbrace \subset H^r(M)$ converging to $A$ in the $r$-H\"older topology and any $\varepsilon >0$, we have
\begin{equation*}
 \mu \Big(  \Big\lbrace x\in M; \; d(E^{u,A_k}_{x}, E^{u,A}_{x})<\varepsilon \quad \textrm{and}\quad d(E^{s,A_k}_{x}, E^{s,A}_{x})<\varepsilon \Big\rbrace \Big) \xrightarrow{k\rightarrow \infty}1.
\end{equation*}

Thus, our main result is the following

\begin{theorem}\label{theorem: continuity of oseledets subspaces}
If $A \in H^r(M)$ is a fiber-bunched cocycle with $\lambda ^+(A,\mu)>\lambda ^-(A,\mu)$ then it is a continuity point for the Oseleteds decomposition with respect to the measure $\mu$.
\end{theorem}

The hypotheses that $A$ is fiber-bunched and $\mu$ has local product structure are only used to apply the results about continuity of Lyapunov exponents from \cite{BBB}. Thus, more generally, if we have a sequence $\lbrace (A_k)_{k\in \mathbb{N}}\rbrace \subset H^r(M)$ \textit{converging uniformly with holonomies} to $A$ as in the main theorem of  \cite{BBB}, then 
\begin{equation*}
 \mu \Big(  \Big\lbrace x\in M; \; d(E^{u,A_k}_{x}, E^{u,A}_{x})<\varepsilon \quad \textrm{and}\quad d(E^{s,A_k}_{x}, E^{s,A}_{x})<\varepsilon \Big\rbrace \Big) \xrightarrow{k\rightarrow \infty}1.
\end{equation*}
Consequently, our result also applies if we restrict ourselves to the space of locally constant cocycles endowed with the uniform topology.

\section{Proof of the theorem}

Let us consider the \textit{projective cocycle} $F_{A}:M\times \mathbb{P}(\mathbb{R}^2)\to M\times \mathbb{P}(\mathbb{R}^2)$ associated to $A$ and $f$ which is given by 
\begin{displaymath}
F_{A} (x,v)=(f(x),\P A(x)v)
\end{displaymath}
where $\P A$ denotes the action of $A$ on the projective space. We say that an $F_A$-invariant measure $m$ on $M\times \mathbb{P}(\mathbb{R}^2)$ \textit{projects} to $\mu$ if $\pi _{\ast}m=\mu$ where $\pi :M\times \mathbb{P}(\mathbb{R}^2) \to M$ is the canonical projection on the first coordinate. Given a non-zero element $v\in \mathbb{R}^2$ we are going to use the same notation to denote its equivalence class in $\P(\mathbb{R}^2)$.

Let $\mathbb{R}^2=E^{s,A}_{x}\oplus E^{u,A}_{x}$ be the Oseledets decomposition associated to $A$ at the point $x\in M$. Consider also 

\begin{displaymath}
m^s=\int _{M}\delta _{(x,E^{s,A}_{x})} d\mu(x)
\end{displaymath}
and
\begin{displaymath}
m^u=\int _{M}\delta _{(x,E^{u,A}_{x})} d\mu(x)
\end{displaymath}
which are $F_A$-invariant probability measures on $M\times \mathbb{P}(\mathbb{R}^2)$ projecting to $\mu$. Moreover, by the Birkhoff ergodic theorem and \eqref{eq: extremal Lyapunov} we have that

\begin{displaymath}
\lambda ^-(A,\mu) =\int _{M\times \mathbb{P}(\mathbb{R}^2)} \varphi _{A}(x,v) dm^s (x, v)
\end{displaymath}
and
\begin{displaymath}
\lambda ^+(A,\mu) =\int _{M \times \mathbb{P}(\mathbb{R}^2)} \varphi _{A}(x,v) dm^u (x,v)
\end{displaymath}
where $\varphi _{A}: M\times \mathbb{P}(\mathbb{R}^2)\rightarrow \mathbb{R}$ is given by
\begin{equation*}
\varphi _{A} (x,v)= \log \dfrac{\parallel A(x)v\parallel}{\parallel v\parallel}.
\end{equation*}

By the (non-uniform) hyperbolicity of $(A,\mu)$ we have the following.

\begin{lemma}\label{lemma:convex combination}
Let $m$ be a probability measure on $M\times \mathbb{P}(\mathbb{R}^2)$ that projects down to $\mu$. Then, $m$ is $F_{A}$-invariant if and only if it is a convex combination of $m^s$ and $m^u$ for some $f$-invariant functions $\alpha ,\beta :M\to [0,1]$ such that $\alpha(x)+\beta (x)=1$ for every $x\in M$.
\end{lemma}
\begin{proof}
One implication is trivial. For the converse one only has to note that every compact subset of $\mathbb{P}(\mathbb{R}^2)$ disjoint from $\lbrace E^u, E^s\rbrace$ accumulates on $E^u$ in the future and on $E^s$ in the past.
\end{proof}

\begin{proof}[Proof of Theorem \ref{theorem: continuity of oseledets subspaces}]

Suppose that $A$ is a fiber-bunched cocycle such that $\lambda ^+(A,\mu)>\lambda ^-(A,\mu)$. As the subset of fiber-bunched elements of $H^r(M)$ is open we may assume without loss of generality that $A_k$ is fiber-bunched for every $k\in \mathbb{N}$. Moreover, since the Lyapunov exponents depend continuously on the cocycle $A$ (see Theorem 1.1 from \cite{BBB}) and $\lambda ^+(A,\mu)>\lambda ^-(A,\mu)$ we may also assume that $\lambda ^+(A_k,\mu)>\lambda ^-(A_k,\mu)$ for every $k\in \mathbb{N}$. We will prove just the assertion about the unstable spaces, that is, that $\mu \left( \left\lbrace x\in M; \; d(E^{u,A_k}_{x}, E^{u,A}_{x})<\delta \right\rbrace \right) \xrightarrow{k\rightarrow \infty}1$. The case of the stable spaces is analogous.

For each $k\in \mathbb{N}$, let us consider the measure
\begin{displaymath}
m_k=\int _{M} \delta _{(x,E^{u,A_k}_{x})} d\mu(x)
\end{displaymath}
and let $m^u$ be the measure given by
\begin{displaymath}
m^u=\int _{M} \delta _{(x,E^{u,A}_{x})} d\mu(x).
\end{displaymath}

These are $F_{A_k}$ and $F_A$-invariant probability measures on $M\times \P(\mathbb{R}^2)$ respectively, projecting to $\mu$. Moreover, $m_k\xrightarrow{k\rightarrow \infty} m^u$. Indeed, let $(m_{k_j})_{j\in \mathbb{N}}$ be a convergent subsequence of $(m_{k})_{k\in \mathbb{N}}$ and suppose that it converges to $\eta$. Since for each $j\in \mathbb{N}$ the measure $m_{k_j}$ is $F_{A_{k_j}}$-invariant and projects to $\mu$ it follows that $\eta$ is an $F_{A}$-invariant measure projecting to $\mu$. Moreover, since  
\begin{displaymath}
\lambda ^+(A_{k_j},\mu) \xrightarrow{j\rightarrow \infty}\lambda ^+(A, \mu)
\end{displaymath}
once the Lyapunov exponents are continuous as functions of the cocycle (see \cite{BBB}) and
\begin{displaymath}
\lambda ^+(A_{k_j},\mu)=\int _{M\times \mathbb{P}(\mathbb{R}^2)} \varphi _{A_{k_j}}dm_{k_j} \xrightarrow{j\rightarrow \infty} \int _{M \times \mathbb{P}(\mathbb{R}^2)} \varphi _{A}d\eta
\end{displaymath}
we get that
\begin{displaymath}
\lambda ^+(A, \mu) = \int _{M \times \mathbb{P}(\mathbb{R}^2)} \varphi _{A}d\eta.
\end{displaymath}
Thus, invoking Lemma \ref{lemma:convex combination} and using the fact that $\mu$ is ergodic it follows that $\eta=m^u$. Indeed, otherwise we would have $\eta=\alpha m^s + \beta m^u$ with $\alpha >0$ and consequently 
$$\int _{M \times \mathbb{P}(\mathbb{R}^2)} \varphi _{A}d\eta= \alpha \lambda ^-(A, \mu) + \beta \lambda ^+(A, \mu) <\lambda ^+(A, \mu).$$
Therefore, $m_k\xrightarrow{k\rightarrow \infty} m^u$ as claimed.

Let $g:M \rightarrow \mathbb{P}(\mathbb{R}^2)$ be the measurable map given by 
\begin{displaymath}
g(x)=E^{u, A}_{x}. 
\end{displaymath} 
Note that its graph has full $m^u$-measure. By Lusin's Theorem, given $\varepsilon >0$ there exists a compact set $K\subset M$ such that the restriction $g_K$ of $g$ to $K$ is continuous and $\mu (K)>1-\varepsilon$. Now, given $\delta >0$, let $U\subset M\times \mathbb{P}(\mathbb{R}^2)$ be an open neighborhood of the graph of $g_K$ such that
\begin{displaymath}
U\cap (K\times \mathbb{P}(\mathbb{R}^2))\subset U_{\delta}
\end{displaymath}
where
\begin{displaymath}
U_{\delta}:=\lbrace (x,v)\in K\times \mathbb{P}(\mathbb{R}^2); \; d(v, g(x))<\delta \rbrace .
\end{displaymath}
By the choice of the measures $m_k$,
\begin{equation}\label{eq: auxiliary eq 1 theo 2}
m_k(U_{\delta})= \mu (\lbrace x\in K; \; d(E^{u,A_k}_{x}, E^{u,A}_{x})<\delta \rbrace ).
\end{equation}
Now, as $m_k\xrightarrow{k\rightarrow \infty} m^u$ it follows that $\liminf m_k(U)\geq m^u(U)> 1-\varepsilon$. On the other hand, as $m_k(K\times \mathbb{P}(\mathbb{R}^2))=\mu(K)> 1-\varepsilon$ for every $k\in \mathbb{N}$, it follows that 
\begin{equation}\label{eq: auxiliary eq 2 theo 2}
m_k(U_{\delta})\geq m_k(U \cap (K\times \mathbb{P}(\mathbb{R}^2)))\geq 1-2\varepsilon
\end{equation}
for every $k$ large enough. Thus, combining \eqref{eq: auxiliary eq 1 theo 2} and \eqref{eq: auxiliary eq 2 theo 2}, we get that $\mu(\lbrace x\in M; \; d(E^{u,A_k}_{x}, E^{u,A}_{x})<\delta \rbrace)\geq 1-2\varepsilon$ for every $k$ large enough completing the proof of Theorem \ref{theorem: continuity of oseledets subspaces}. 

\end{proof}

\end{document}